\documentclass{gtpart}     
\usepackage[english]{babel}
\usepackage[latin1]{inputenc}
\usepackage{amsfonts}
\usepackage{graphicx}
\usepackage[all]{xy}
\usepackage{epsfig}

\title{Homeomorphism and Homotopy Types of Restricted Configuration Spaces of Metric Graphs}

\author{James Dover}
\givenname{James}
\surname{Dover}
\address{Department of Mathematical Sciences, Cameron University, Burch Hall, Room B001, 2800 W. Gore Blvd., Lawton, OK, USA 73505}
\email{jdover@cameron.edu}
\urladdr{}

\author{Murad \"{O}zayd\i n}
\givenname{Murad}
\surname{\"{O}zayd\i n}
\address{Department of Mathematics, University of Oklahoma, 601 Elm Avenue, Norman, OK, USA 73019-3103}
\email{mozaydin@math.ou.edu}
\urladdr{}

\arxivreference{}
\arxivpassword{}

\volumenumber{}
\issuenumber{}
\publicationyear{}
\papernumber{}
\startpage{}
\endpage{}
\doi{}
\MR{}
\Zbl{}
\received{}
\revised{}
\accepted{}
\published{}
\publishedonline{}
\proposed{}
\seconded{}
\corresponding{}
\editor{}
\version{}

\newcommand{\vect}[1]{\textbf{#1}}

\newtheorem{theorem}{Theorem}[section]
\newtheorem{lemma}[theorem]{Lemma}

\newtheorem{cor}[theorem]{Corollary}

\newtheorem{exa}[theorem]{Example}

\begin{document}

\begin{abstract}    
For $\Gamma$ a finite, connected metric graph, $\Gamma^n_{\vect{r}}$ is the space of configurations of $n$ points in $\Gamma$ with a restraint parameter $\vect{r}$ dictating the minimum distance allowed between each pair of points.  These restricted configuration spaces come up naturally in topological robotics.  In this paper, we study the homotopy, homeomorphism, and isotopy types of $\{\Gamma^n_{\vect{r}}\}$ over the space of parameters $\vect{r}$ and provide a polynomial upper bound (in the number of edges of $\Gamma$) for the number of isotopy types.
\end{abstract}

\maketitle


\section{Introduction}

We study compact subspaces of the configuration space (of $n$ ordered points) of a finite {\it metric} graph, parametrized by restricting the proximities of the points.  These are variations on the {\it configuration spaces of thick particles on a metric graph} of Deeley \cite{Deeley11a}, \cite{Deeley11b}.  

The $n$th configuration space of (a topological space) $\Gamma$, which we denote by $\Gamma^{\underline{n}}$, is the space of one-to-one functions from $\{1, \ldots, n\}$ into $\Gamma$, that is: \begin{equation*}\Gamma^{\underline{n}} := \{(x_1, \ldots , x_n) \in \Gamma^n \, | \, x_i \neq x_j, \, 1 \leq i < j \leq n\}.\end{equation*}  Configuration spaces  are important in topological robotics (see, for example, \cite{Abrams00}, \cite{AG02}, \cite{Farber08}, \cite{Ghrist01}, \cite{Ghrist10}, \cite{GK02}), where the points correspond to robots  or AGVs (Automated Guided Vehicles) and two AGVs are not allowed to occupy the same position.     It seems more reasonable to require that any two points are no closer than a given threshold, yielding a configuration space of hard disks \cite{BBK11}, \cite{CGKM12}, or thick particles.  (From a topological viewpoint, these spaces have the additional benefit of being compact when $\Gamma$ is compact.)  When the robots move on a one-dimensional network of tracks, the underlying space $\Gamma$ is a graph.    In order to quantify the proximity restraints, we need a notion of distance on the graph.

A \textbf{metric graph} $\Gamma$ is obtained by assigning (positive) lengths to the edges (closed $1$--cells) of a graph (i.e., one-dimensional CW complex).  This gives an isometry between each closed edge and an interval and thus defines a path metric $\delta$ on $\Gamma$.  Our graphs will always be finite and connected (each connected component can be analyzed separately).  We will also assume that the underlying topological space is not a single point or homeomorphic to either the unit interval $[0,1]$ or the unit circle $S^1$, as these cases are well understood.  Then there is a unique cell structure with a minimum number of cells on our graph so that no node ($0$--cell) has valence two.  It is convenient for us to alter this slightly by subdividing loops at their midpoints; this yields the minimal {\it regular} (i.e., each closed cell is homeomorphic to a closed disk) cell structure on $\Gamma$.

Let $n \geq 2$ be an integer and $\vect{r} = (r_{ij})_{1 \leq i < j \leq n}$ be in $\R^{\binom{n}{2}}$.  The space \begin{equation*}\Gamma^n_{\vect{r}} := \{(x_1, \ldots, x_n) \in \Gamma^n \, | \, \delta(x_i,x_j) \geq r_{ij} \text{ for } 1 \leq i < j \leq n\}\end{equation*} is the $n$th configuration space of ``thick particles'' with restraint parameter $\vect{r}$.  
We will also consider (affine) subspaces of the space $\R^{\binom{n}{2}}$ of  restraint parameters such as the $n$--dimensional subspace given by $r_{ij} = r_i + r_j$ where the $i$th particle has radius $r_i$ or the $1$--dimensional subspace given by $r_{ij} = r$ for a single scalar parameter $r$ (i.e., all particles have the same radius/thickness $\frac{r}{2}$) considered by Deeley in \cite{Deeley11a} and \cite{Deeley11b}.

 Deeley shows, in particular, that the infinite family of spaces $\{\Gamma^2_r\}_{r \geq 0}$ consists of only finitely many homotopy types, and he gives an upper bound which is exponential in $E$, the number of edges of the graph.  He raises the question whether the number of homotopy types can be exponential in $E$.  Actually, there is a quadratic (in $E$) upper bound for the number of homotopy types of $\{\Gamma^2_r\}$.

We show (Theorem \ref{maintheorem} below) that there is an upper bound for the number of {\it homeomorphism} (in fact, {\it combinatorial isotopy}) types which is: (1) quadratic (in the number of edges) for \{$\Gamma^2_r\}$, (2) a polynomial of degree $n$ for $\{\Gamma^n_r\}$, and (3) a polynomial of degree $n d$ for $\{\Gamma^n_{\vect{r}}\}$, where the restraint parameters $\vect{r}$ lie in a $d$--dimensional subspace of $\R^{\binom{n}{2}}$ above.  Conversely, in Example \ref{corolla}, an infinite family of metric graphs provides a quadratic lower bound for the number of homotopy types of $\{\Gamma^2_r\}$.
 
The strategy of the proof is to describe a regular cell structure on $\Gamma^n_{\vect{r}}$ and an equivalence relation on the space of parameters $\vect{r}$.  If $\vect{r}$ is equivalent to $\vect{s}$, then $\Gamma^n_{\vect{r}}$ is isotopic to $\Gamma^n_{\vect{s}}$ in $\Gamma^n$ (Theorem \ref{Closed-Isotopy}).
The equivalence classes correspond to certain faces of a (finite) real hyperplane arrangement in the $d$--dimensional space of parameters.  The number of hyperplanes is bounded above by a polynomial (of degree $n$) in the number of edges of $\Gamma$.  Since the number of faces of a real hyperplane arrangement has an upper bound which is a polynomial (of degree $d$) in the number of hyperplanes, Theorem \ref{maintheorem} follows.

In fact, our setup handles the richer framework in which the edge lengths can also vary, so we have a parameter space of dimension $d + E$.  If we denote by $\Gamma^n_{\vect{L},\vect{r}}$ the configuration space with restraint parameter $\vect{r}$ of the metric graph $\Gamma$ whose edge lengths are given by the vector $\vect{L}$, the number of homeomorphism types of $\{\Gamma^n_{\vect{L},\vect{r}}\}$ is again bounded by the number of faces of a real hyperplane arrangement in $\R^{d + E}$.  The number of hyperplanes in this arrangement is still bounded by a polynomial in $E$ of degree $n$.   

We start with a convenient (regular) cell structure on $\Gamma^n$ whose cells are products of simplices (an edge with $k$ points on it contributing a $k$--dimensional simplex) which encode the relative positions of the points on each edge (see the beginning of Section \ref{cells} for the precise definition).  There are  $E^{\bar{n}}=E(E+1)\cdots(E+n-1)$ maximal ($n$--dimensional) cells.  Each maximal cell $c_{\vect{r}}$ of $\Gamma^n_{\vect{r}}$ is the intersection of $\Gamma^n_{\vect{r}}$ with an $n$--cell $c$ of $\Gamma^n$ (because every node lies on an edge).  The cell $c_{\vect{r}}$ is naturally a convex polytope of the form $\{\vect{x} \in \R^n \, | \, A\vect{x} \leq \vect{b}\}$ where $\vect{b} = \vect{b}(\vect{r})$ depends both on $c$ and on the parameter $\vect{r}$ while the $m \times n$ matrix $\textbf{A}$ depends only on the $n$--cell $c$ of $\Gamma^n$, as explained in Section \ref{cells}.  The cells of $\Gamma^n_{\vect{r}}$ are all of the faces of the maximal cells $\{c_{\vect{r}}\}$.

We study the combinatorial types of parametric polytopes  via a Galois connection in Section \ref{polytopes} below.   In Section \ref{cells}, we apply these considerations to the cells $c_{\vect{r}}$, and we give a combinatorial model determining the isotopy type of $\Gamma^n_{\vect{r}}$.  In Section \ref{example}, we compare our setup to Deeley's \cite{Deeley11b}, analyze the cells of $\Gamma^2_r$ and compute the explicit quadratic upper bound $9E^2 - 5E - 1$ for the number of isotopy types of $\{\Gamma^2_r\}$.  We establish the polynomial upper bounds for the number of isotopy types of $\{\Gamma^n_{\vect{r}}\}$ in Section \ref{bound} via a hyperplane arrangement in a $d$--dimensional parameter space.  Finally, in Section \ref{strict}, we consider the topology of (non-compact) restricted configuration spaces where some of the proximity restraints are given by strict inequalities.

\section{Parametric Polytopes}\label{polytopes}

A parametric polytope in $\R^n$ is given by a system of inequalities
\begin{equation*} \vect{A}x \leq \vect{b}\end{equation*} where  $x$ is (a column vector) in $\R^n$, $\Lambda$ is a finite set,  $\vect{b} = (b_{\lambda})_{\lambda \in \Lambda}$ is in $\R^{\Lambda}$, and the rows of $\vect{A}$ are denoted by $A_{\lambda}$.  For a fixed $\vect{b} \in \R^{\Lambda}$, this system of inequalities indexed by $\Lambda$ determines some convex polytope $c_{\vect{b}} \subset \R^n$.   We will further assume that for all $\vect{b} \in \R^{\Lambda}$, the polytope $c_{\vect{b}}$ is bounded.  We denote by $H_{\lambda}$ the hyperplane in $\R^n$ defined by the equation $A_{\lambda}  x = b_{\lambda}$.  

For any subset $\beta \subset \Lambda$, $\vect{A}_{\beta}$ is the submatrix of $\vect{A}$ with rows $A_{\lambda}, \lambda \in \beta$, and $\vect{b}_{\beta}$ is $(b_{\lambda})_{\lambda \in \beta}$.
If $\vect{A}_{\beta}$ is invertible (we will refer to such $\beta$ as {\it basic}), the unique solution $\vect{A}^{-1}_{\beta}\vect{b}_{\beta}$ of the system of equations $\{A_{\lambda} x = b_{\lambda}\}$ is a {\it potential vertex} of $c_{\vect{b}}$.  If a potential vertex $\vect{A}^{-1}_{\beta}\vect{b}_{\beta}$ satisfies the remaining inequalities indexed by $\Lambda$, then it is a vertex of $c_{\vect{b}}$.  In this way, the vertices of $c_{\vect{b}}$ are given by subsets $\beta \subset \Lambda$ such that $\{A_{\lambda}\}_{\lambda \in \beta}$ forms a basis.  

Define $S_{\vect{b}} : = \{\beta \, | \, \vect{A}_{\beta} \text{ a basis of } \R^n, \, A_{\lambda}  \vect{A}^{-1}_{\beta}\vect{b}_{\beta} \leq b_{\lambda} \, \forall \, \lambda \in \Lambda \}$.  This is the collection of basic sets which give rise to actual vertices of $c_{\vect{b}}$.  

However, two such basic sets $\alpha$ and $\beta$ produce the same potential vertex if $\vect{A}_{\alpha}^{-1}\vect{b}_{\alpha} = \vect{A}_{\beta}^{-1}\vect{b}_{\beta}$.   Define $v_{\vect{b}}(\beta) : = \{\lambda \, | \, A_{\lambda} \vect{A}^{-1}_{\beta}\vect{b}_{\beta} = b_{\lambda}\}$.  This is the maximal set such that any basic subset produces the same potential vertex as $\beta$.  If $\vect{A}_{\alpha}^{-1}\vect{b}_{\alpha} = \vect{A}_{\beta}^{-1}\vect{b}_{\beta}$, then $v_{\vect{b}}(\alpha) = v_{\vect{b}}(\beta)$.  If $\vect{A}_{\beta}^{-1}\vect{b}_{\beta}$ is a vertex of $c_{\vect{b}}$, then $v_{\vect{b}}(\beta)$ is the \textbf{abstract vertex} that indexes it.

The \textbf{type} of $\vect{b}$ is the collection $T_{\vect{b}} : = \{v_{\vect{b}}(\beta) \, | \, \beta \in S_{\vect{b}}\}$ of abstract vertices  of $c_{\vect{b}}$.

\begin{lemma}\label{polytope}
The family of polytopes $\{c_{\vect{b}}\}_{\vect{b} \in \R^{\Lambda}}$ has only finitely many combinatorial types.
\end{lemma}

\begin{proof}
We will show that the intersection semi-lattice $\{\bigcap_{\beta \in S} v_{\vect{b}}(\beta) \, | \, \emptyset \neq S \subseteq S_{\vect{b}}\}$ of $T_{\vect{b}}$ is isomorphic to the face poset $F_{\vect{b}}$ of $c_{\vect{b}}$.   Hence, if $T_{\vect{b}} = T_{\vect{b}^{\prime}}$, then $c_{\vect{b}}$ and $c_{\vect{b}^{\prime}}$ have the same combinatorial type.  Since $\Lambda$ is finite, there are only finitely many possibilities for the abstract vertex set $T_{\vect{b}}$.

Let $V = V_{\vect{b}} := \{\vect{A}^{-1}_{\beta}\vect{b}_{\beta} \, | \, \beta \in S_{\vect{b}}\}$ be the collection of actual vertices of $c_{\vect{b}}$.  

Define:
\begin{equation*}
 g : \{\gamma \subseteq \Lambda \; | \; \gamma \subseteq v_{\vect{b}}(\beta) \text{ for some } \beta \in S_{\vect{b}}\} \rightarrow 2^V \setminus \{\emptyset\}
\end{equation*}
\begin{equation*}
g(\gamma) := \bigcup_{\beta : \; \gamma \subseteq v_{\vect{b}}(\beta) \in T_{\vect{b}}} \{\vect{A}^{-1}_{\beta}\vect{b}_{\beta}\}
\end{equation*}
and
\begin{equation*}
h:  2^V \setminus \{\emptyset\} \rightarrow  \{\gamma \subseteq \Lambda \; | \; \gamma \subseteq v_{\vect{b}}(\beta) \text{ for some } \beta \in S_{\vect{b}}\}
\end{equation*}
\begin{equation*}
h(\rho) := \bigcap_{\beta : \; \vect{A}^{-1}_{\beta}\vect{b}_{\beta} \in \rho} v_{\vect{b}}(\beta)
\end{equation*}

The two maps, $g$ and $h$, are well-defined and order-reversing (with respect to inclusion).  It can be checked that $\rho \subset g(h(\rho))$ for all $\rho$ and that $\gamma \subset h(g(\gamma))$ for all $\gamma$.  Therefore, $g$ and $h$ form a Galois connection (see \cite{Mac Lane98}) between the two posets, and the images of $g \circ h$ and $h \circ g$ are isomorphic.  The image of $g \circ h$ is the poset of sets of vertices that span a face of $c_{\vect{b}}$, ordered by inclusion, which is clearly isomorphic to the face poset $F_{\vect{b}}$.  The image of $h \circ g$ is the intersection poset of $T_{\vect{b}}$.

\end{proof}

To summarize, the vertices of a bounded, convex polytope are indexed by subsets $v_{\vect{b}}(\beta)$ of $\Lambda$, the abstract vertices.  If $v_{\vect{b}}(\alpha)$ and $v_{\vect{b}}(\beta)$ index two vertices of $c_{\vect{b}}$, then the intersection $v_{\vect{b}}(\alpha) \cap v_{\vect{b}}(\beta)$ indexes the lowest dimensional face of $c_{\vect{b}}$ containing both vertices.  Hence, the faces of $c_{\vect{b}}$ are also indexed by subsets of $\Lambda$ (via an order-reversing isomorphism).  If $T_{\vect{b}} = T_{\vect{b}^{\prime}}$ for $b, b^{\prime} \in \R^{\Lambda}$, then not only does $c_{\vect{b}}$ clearly have the same combinatorial type as $c_{\vect{b}^{\prime}}$, but the isomorphism is such that corresponding faces are indexed by the same subset of $\Lambda$.

\begin{lemma}
Let $\vect{b}^{\prime} \in \R^{\Lambda}$.  If $\emptyset \neq T_{\vect{b}^{\prime}} = T$, then the set of all $\vect{b} \in \R^{\Lambda}$ with $T_{\vect{b}} = T$ is convex.
\end{lemma}
\begin{proof}
Suppose $T = T_{\beta^{\prime}} \neq \emptyset$.  Define 
\begin{equation*}
P_T := \bigcap_{\beta \subseteq v \in T} \left\{ \vect{b} \; | \begin{array}{lr}
 A_{\lambda} \vect{A}_{\beta}^{-1}\vect{b}_{\beta} = b_{\lambda}, & \lambda \in v \\
A_{\mu} \vect{A}^{-1}_{\beta}\vect{b}_{\beta} < b_{\mu}, & \mu \notin v
\end{array}\right \}.
\end{equation*}
This is a convex subset of $\R^{\Lambda}$.  We claim that $\vect{b} \in P_T$ if and only if $T_{\vect{b}} = T$.

If $T_{\vect{b}} = T$, then, if $v \in T$ and $\beta$ is a basic subset with $v_{\vect{b}}(\beta) = v$, then $A_{\lambda} \vect{A}^{-1}_{\beta}\vect{b}_{\beta} = b_{\lambda}$ for all $\lambda \in v$ by the definition of $v_{\vect{b}}(\beta)$.  Also, $A_{\mu} \vect{A}^{-1}_{\beta}\vect{b}_{\beta} < b_{\mu}$ for all $\mu \notin v$ since $v \in T_{\vect{b}}$ (i.e., $\vect{A}^{-1}_{\beta}\vect{b}_{\beta}$ is an actual vertex of $c_{\vect{b}}$).  Therefore, we have that $\vect{b} \in P_T$.

Now let $\vect{b} \in P_T$.  First we note that $S_{\vect{b}^{\prime}} \subseteq S_{\vect{b}}$:  If $\beta$ produces an actual vertex of $c_{\vect{b}^{\prime}}$, then it produces an actual vertex of $c_{\vect{b}}$.  Further, for all $\gamma \in S_{\vect{b}^{\prime}}$, we have $v_{\vect{b}}(\gamma) = v_{\vect{b}^{\prime}}(\gamma)$.  Therefore, if we can show that $S_{\vect{b}} = S_{\vect{b}^{\prime}}$, we will also have that $T_{\vect{b}} = T_{\vect{b}^{\prime}}$.

Suppose there exists $\beta \in S_{\vect{b}} \setminus S_{\vect{b}^{\prime}}$.  Since $c_{\vect{b}}$ is connected, we can find $\alpha \in S_{\vect{b}^{\prime}}$ such that the vertices $\vect{A}^{-1}_{\alpha}\vect{b}_{\alpha}$ and $\vect{A}^{-1}_{\beta}\vect{b}_{\beta}$ lie on an edge of $c_{\vect{b}}$.  In terms of the abstract vertices, this means $v_{\vect{b}}(\alpha) \neq v_{\vect{b}}(\beta)$ and $|\alpha \cap \beta| = n - 1$.  Say $\beta = (\alpha \cup \{\mu\}) \setminus \{\lambda\}$.

Now $\vect{A}^{-1}_{\alpha}\vect{b}^{\prime}_{\alpha}$ is a vertex of $c_{\vect{b}^{\prime}}$, but $\vect{A}^{-1}_{\beta}\vect{b}^{\prime}_{\beta}$ is not.  Therefore, it must be that $A_{\nu} \vect{A}^{-1}_{\beta}\vect{b}^{\prime}_{\beta} > b^{\prime}_{\nu}$ for some $\nu \in \Lambda$.  Also, the line segment between the points $\vect{A}^{-1}_{\alpha}\vect{b}^{\prime}_{\alpha}$ and $\vect{A}^{-1}_{\beta}\vect{b}^{\prime}_{\beta}$ must pass through at least one hyperplane $H_{\nu}: A_{\nu} x = b^{\prime}_{\nu}$ for such $\nu$.  Choose a $\nu$ such that this intersection is closest to $\vect{A}^{-1}_{\alpha}\vect{b}^{\prime}_{\alpha}$.  Then $H_{\nu}$ separates the line segment into two pieces, one lying in $c_{\vect{b}^{\prime}}$ and one not.  Set $\gamma := (\alpha \cap \beta) \cup \{\nu\}$ .  Then $\vect{A}_{\gamma}$ is a basis, and $\gamma \in S_{\vect{b}^{\prime}}$.

Now $A_{\nu} \vect{A}^{-1}_{\beta}\vect{b}_{\beta} < b_{\nu}$ and $A_{\nu} \vect{A}^{-1}_{\beta}\vect{b}^{\prime}_{\beta} > b^{\prime}_{\nu}$, so there exists $\vect{b}^{\prime \prime}$ on the line segment between $\vect{b}$ and $\vect{b}^{\prime}$ such that $A_{\nu} \vect{A}^{-1}_{\beta}\vect{b}^{\prime \prime}_{\beta} = b^{\prime \prime}_{\nu}$.  Hence, $\nu \in v_{\vect{b}^{\prime \prime}}(\beta)$ and $\gamma \subset v_{\vect{b}^{\prime \prime}}(\beta)$, yielding that $v_{\vect{b}^{\prime \prime}}(\beta) = v_{\vect{b}^{\prime \prime}}(\gamma)$.  In particular, $\beta \in v_{\vect{b}^{\prime \prime}}(\gamma)$. However, since $P_T$ is convex, $\vect{b}^{\prime \prime} \in P_T$, so by our earlier remarks,  $v_{\vect{b}^{\prime \prime}}(\gamma) = v_{\vect{b}^{\prime}}(\gamma)$.  This is a contradiction to $\beta \notin S_{\vect{b}^{\prime}}$.  

Thus, $S_{\vect{b}} = S_{\vect{b}^{\prime}}$ and $T_{\vect{b}} = T_{\vect{b}^{\prime}} = T$ as claimed.  This shows that the convex set $P_T$ is exactly the set of all $\vect{b} \in \R^{\Lambda}$ with $T_{\vect{b}} = T$.
\end{proof}

If $\vect{b}$ and $\vect{b}^{\prime}$ have the same (nonempty) type $T$, then any $\vect{b}^{\prime \prime} = (1-t)\vect{b} + t \vect{b}^{\prime}$ with $0 \leq t \leq 1$ is of type $T$ as well.  We observe that any vertex $\vect{A}^{-1}_{\beta}\vect{b}^{\prime \prime}_{\beta}$ of $c_{\vect{b}^{\prime \prime}}$ is simply the combination $(1-t)\vect{A}^{-1}_{\beta}\vect{b}_{\beta} + t\vect{A}^{-1}_{\beta}\vect{b}^{\prime}_{\beta}$ of the corresponding vertices of $c_{\vect{b}}$ and $c_{\vect{b}^{\prime}}$.

\section{The Cell Structure of $\Gamma^n_{\vect{r}}$}\label{cells}

We begin with a regular cell structure on $\Gamma^n$.  The graph $\Gamma$ itself is a regular cell complex (as long as $\Gamma$ has no loops) whose cells are its nodes and edges.  The space $\Gamma^n$ has the obvious product structure: Its cells are $n$--fold products of nodes and edges.  A cell $c$ is identified with the subspace $[0,L_{e_1}] \times \cdots \times [0,L_{e_m}]$ of $\R^m$ where $e_1, \ldots, e_m$ are the edge factors of $c$.   However, if an edge is repeated in $c$, the part of $c$ inside $\Gamma^n_{\vect{r}}$ is usually not connected.  To circumvent this problem, we will subdivide cells which include repeated edges in this product in the following way.  A product $[0,L_e]^k$ can be triangulated with maximal simplices  of the form $\{(x_1, \ldots, x_k) \in [0,L_e]^k \; | \; x_{\sigma(1)} \leq \cdots \leq x_{\sigma(k)}\}$ for a permutation $\sigma$.  Doing this for each repeated edge produces a cell which is a product of simplices.  Now the intersection of $\Gamma^n_{\vect{r}}$ with a cell $c$ will be a (bounded) convex polytope defined by a (finite) system of inequalities.  These polytopes and all their faces will give the cell structure of $\Gamma^n_{\vect{r}}$.

If $c$ is a maximal cell of $\Gamma^n$, then $c$ is a subset of the product $e_1 \times \cdots \times e_n$ specified by permutations assigned to the repeated factors.   
Let $a_i$ and $b_i$ be the nodes of $e_i$ corresponding to $0$ and $L_{e_i}$ respectively.  The following inequalities describe the conditions for a point $(x_1, \ldots, x_n)$ in $c$ to lie in $\Gamma^n_{\vect{r}}$:
\begin{enumerate}
\renewcommand{\labelenumii}{\arabic{enumii}.}
\item[$C^i$:] For $1 \leq i \leq n$,
\begin{enumerate}
\item[$C^i_1$:] $-x_i \leq 0$,
\item[$C^i_2$:] $x_i \leq L_{e_i}$.
\end{enumerate}
\item[$D^{ij}$:] For $1 \leq i < j \leq n$, if $e_i \neq e_j$,
\begin{enumerate}
\item[$D^{ij}_1$:] $-x_i - x_j \leq -r_{ij} + \delta(a_i,a_j)$,
\item[$D^{ij}_2$:] $-x_i + x_j \leq -r_{ij} + \delta(a_i,b_j) + L_{e_j}$,
\item[$D^{ij}_3$:] $x_i - x_j \leq -r_{ij} + \delta(b_i,a_j) + L_{e_i}$,
\item[$D^{ij}_4$:] $x_i + x_j \leq -r_{ij} + \delta(b_i,b_j) + L_{e_i} + L_{e_j}$.
\end{enumerate}
\item[$E^{ij}$:] 
For $1 \leq i , j \leq n$, if $e_i = e_j$ and $x_i \leq x_j$ in $c$,
\begin{enumerate}
\item[$E^{ij}_1$:] $x_i - x_j \leq 0$,
\item[$E^{ij}_2$:] $x_i - x_j \leq -r_{ij}$,
\item[$E^{ij}_3$:] $-x_i + x_j \leq -r_{ij} + \delta(a_i,b_i) + L_{e_i}$.
\end{enumerate}
\end{enumerate}
The intersection of the cell $c$ with $\Gamma^n_{\vect{r}}$ is the (bounded) convex polytope determined by these inequalities, which we denote by $c_{\vect{r}}$.  

In the language of the previous section, for a maximal cell $c$ of $\Gamma^n$, the index set $\Lambda = \Lambda^c$ is the collection of labels, $C^i_{1-2},D^{ij}_{1-4},$ and $E^{ij}_{1-3}$, which are applicable to $c$.  Let $\Omega := \{\vect{r} \in \R^{\binom{n}{2}} \, | \, r_{ij} \geq 0\}$.  Each $b_{\lambda}$ is an affine function from $\Omega$ to $\R$.

The inequalities of types $C^i_1$, $C^i_2$, and $E^{ij}_1$ describe the maximal cell $c$, and replacing a collection $\tau$ of these with equalities describes a lower dimensional face $d$ of $c$ in $\Gamma^n$.  The solution set of those equations and the remaining inequalities is the cell $d_{\vect{r}} \subseteq c_{\vect{r}}$ of $\Gamma^n_{\vect{r}}$.  If we incorporate those equations into the remaining inequalities, some of the variables are eliminated or identified resulting in a collection of inequalities in $\R^{\dim d}$.  At this point, we eliminate some redundant inequalities: For example, if $x_i = 0$, the  inequality resulting from $D^{ij}_3$ is always made redundant by the one resulting from $D^{ij}_1$, and the inequality resulting from $D^{ij}_4$ is always made redundant by the one resulting from $D^{ij}_2$.  These redundancies do not depend on $\vect{r}$.  Thus, for any cell $d$ of $\Gamma^n$, we have a fixed collection of inequalities in $\dim d$ variables which determine the cell $d_{\vect{r}} = d \cap \Gamma^n_{\vect{r}}$;  moreover, this collection is the same regardless of the maximal cell $c$ containing $d$ with which we started.  Furthermore, if $d$ is a face of $c$, each of these inequalities naturally corresponds to the inequality for $c$ from which it is obtained, giving a description of $\Lambda^d$ as a subset of $\Lambda^c$.   Therefore, if a face of $d_{\vect{r}}$ is indexed by a subset $\rho \subset \Lambda^d$, it is identified with the face of $c_{\vect{r}}$ indexed by the set $\rho \cup \tau \subset \Lambda^c$.

The polytopes $c_{\vect{r}}$ and their faces comprise a regular cell structure on $\Gamma^n_{\vect{r}}$.  As described in the proof of \ref{polytope}, the face poset $F^c_{\vect{r}}$ of $c_{\vect{r}}$ is isomorphic to the intersection poset of $T^c_{\vect{r}}$, which is a collection of subsets of $\Lambda^c$, the abstract vertices of $c_{\vect{r}}$.   If $d$ is a face of $c$, then the intersection poset of $T^d_{\vect{r}}$ naturally includes into that of $T^c_{\vect{r}}$ as a principal ideal.

We define an equivalence relation on $\Omega$ by: $\vect{r} \sim \vect{s}$ if and only if $T^c_{\vect{r}} = T^c_{\vect{s}}$ for every cell $c$ of $\Gamma^n$.  Because $\Gamma^n$ has only finitely many cells and there are only finitely many possibilities for $T^c_{\vect{r}}$ as $\vect{r}$ varies, there are only finitely many equivalence classes in $\Omega$.

\begin{theorem}\label{Closed-Homeo}
If $\vect{r} \sim \vect{s}$, then $\Gamma^n_{\vect{r}}$ is homeomorphic to $\Gamma^n_{\vect{s}}$.
\end{theorem}

\begin{proof} Suppose $\vect{r} \sim \vect{s}$.  Then, for every cell $c$ of $\Gamma^n$, $T^c_{\vect{r}} = T^c = T^c_{\vect{s}}$.  We will show that the regular cell structures on $\Gamma^n_{\vect{r}}$ and $\Gamma^n_{\vect{s}}$ have isomorphic face posets.  This implies our result because the topology of a regular cell complex $X$ is determined by its face poset $F(X)$: $X$ is homeomorphic to its barycentric subdivision, i.e., the geometric realization of the order complex of F(X) (\cite{LW69} Ch 3 \S 1 and \cite{BLSWZ99} p200).

Let $F_{\vect{r}}$ denote the full face poset of the regular cell structure on $\Gamma^n_{\vect{r}}$.  For each cell $c$ of $\Gamma^n$, we have a poset isomorphism $\phi^c : F^c_{\vect{r}} \rightarrow F^c_{\vect{s}}$, which maps a face of $c_{\vect{r}}$ indexed by the set $\rho \subset \Lambda^c$ to the face of $c_{\vect{s}}$ indexed by $\rho$.  We will patch all of these maps together to produce an isomorphism $\phi : F_{\vect{r}} \rightarrow F_{\vect{s}}$.  

Let $f$ be a face of $\Gamma^n_{\vect{r}}$.  Let $c$ be the minimal cell of $\Gamma^n$ that contains $f$.  Then $f$ is a face of $c_{\vect{r}}$.  Define $\phi(f) := \phi^c(f)$.  This is well-defined: If $f$ is contained in two cells of the same dimension, it is also contained in their intersection, which is made up of lower dimensional cells, so there is a unique minimal cell containing $f$.  

We must show that $\phi$ is order-preserving.  First, we observe that if $d \subset c$ in $\Gamma^n$, then $\phi^c$ restricts to $\phi^d$ on the subposet $F^d_{\vect{r}}$ of $F^c_{\vect{r}}$.  Let $\tau \subset \Lambda^c$ be the index for the face of $c$ corresponding to $d$.  Then let $\rho \subset \Lambda^d$ be the index for some face $f$ of $d_{\vect{r}}$.  Then $\phi^d(f)$ is the face of $d_{\vect{s}}$ indexed by $\rho$.  Including this face into $F^c_{\vect{s}}$ gives the index $\rho \cup \tau$, which is the same as the index for $f$ in $F^c_{\vect{r}}$, meaning that $\phi^c(f) = \phi^d(f)$.

It is clearly one-to-one and onto since its inverse is defined in the same way using the inverses of each $\phi^c$.
\end{proof}

When $\vect{r} \sim \vect{s}$, we will denote by $\phi^{\vect{r}}_{\vect{s}}$ the homeomorphism from $\Gamma^n_{\vect{r}}$ to $\Gamma^n_{\vect{s}}$.

\begin{theorem}\label{Closed-Isotopy}
If $\vect{r} \sim \vect{s}$, then $\Gamma^n_{\vect{r}}$ is isotopic to $\Gamma^n_{\vect{s}}$ in $\Gamma^n$.
\end{theorem}
\begin{proof}
We define the function $\Phi \colon \Gamma^n_{\vect{r}} \times [0,1] \rightarrow \Gamma^n$ by $\Phi(x,t) := (1-t)x + t \phi^{\vect{r}}_{\vect{s}}(x)$.  Clearly, $\Phi(-,0)$ is the inclusion of  $\Gamma^n_{\vect{r}}$ into $\Gamma^n$, and $\Phi(-,1)$ is the homeomorphism $\phi^{\vect{r}}_{\vect{s}}$.  It remains to show that each $\Phi(-,t)$ is one-to-one.

For each maximal cell $c$ of $\Gamma^n$, $c_{\vect{r}}$ and $c_{\vect{s}}$ are homeomorphic via the restriction $\phi^{\vect{r}}_{\vect{s}} |_{c_{\vect{r}}}$.  If $\vect{t} := (1-t)\vect{r} + t\vect{s}$, we have that $c_{\vect{r}}$ (if it is nonempty) is homeomorphic to $c_{\vect{t}}$ via the map $x \mapsto (1-t)x + t\phi^{\vect{r}}_{\vect{s}} |_{c_{\vect{r}}}(x)$ because that is the map for the vertices.

By piecing together these maps $x \mapsto (1-t)x + t\phi^{\vect{r}}_{\vect{s}} |_{c_{\vect{r}}}(x)$ for all maximal cells $c$ for which $c_{\vect{r}} \neq \emptyset$, we can produce a one-to-one map which we will call $\phi^{\vect{r}}_{\vect{t}}$.  We remark that it is not necessarily true that $\vect{t} \sim \vect{r}$; there may be a cell $c$ for which $c_{\vect{r}}$ is empty but $c_{\vect{t}}$ is not.  If $\vect{t} \sim \vect{r}$, the map we just defined is the same as the homeomorphism $\phi^{\vect{r}}_{\vect{t}}$ given by the combinatorial equivalence.

Finally, it is a simple observation that $\Phi(x,t) = \phi^{\vect{r}}_{\vect{t}}(x)$ for any $x \in \Gamma^n_{\vect{r}}$.
\end{proof}

Here we make a side remark about the issue of convexity of an equivalence class $[\vect{r}]$ in $\Omega$.  For each cell $c$ of $\Gamma^n$, the set of parameters $\vect{r}$ yielding a specific nonempty type for $c_{\vect{r}}$ is convex.  However, some $c_{\vect{r}}$ will be empty.  If $\vect{r} \sim \vect{s}$, then the nonempty cells will not change type  along the line segment between $\vect{r}$ and $\vect{s}$ in $\Omega$.  However, if $c_{\vect{r}}$ is empty, it may be happen that, for some $t \in (0,1)$, $c_{(1-t)\vect{r} + t \vect{s}}$ is nonempty.  On the other hand, if $\vect{r} \sim \vect{s}$ and $s_{ij} \geq r_{ij}$ for all $\{i,j\}$, then $(1-t)\vect{r} + t \vect{s} \sim \vect{r}$ for all $t \in [0,1]$ because $c_{(1-t)\vect{r} + t \vect{s}} \subseteq c_{\vect{r}}$ for all cells $c$ of $\Gamma^n$.

\section{Example: The Cells of $\Gamma^2_r$}\label{example}

Here we consider the case of $\Gamma^2_r = \{(x,y) \in \Gamma^2 \, | \, \delta(x,y) \geq r\}$ to illustrate our approach in contrast to that of Deeley in \cite{Deeley11b}.  We remark first that Deeley does not allow multiple edges between a pair of nodes of $\Gamma$.  Further, he subdivides the graph $\Gamma$ so that each edge gives the shortest distance between its endpoints.  Our approach allows multiple edges and requires no subdivision of the graph.  Finally, Deeley must subdivide the $2$--cells of $\Gamma^2$ in order to achieve a uniform distance function on each cell.  For us, the distance is the minimum of four functions.  

We realize the cells of $\Gamma^2_r$ as parametric polytopes.  There are two types of maximal cells of $\Gamma^2$: cells where the points $x$ and $y$ lie on the same edge of $\Gamma$ and cells where they lie on distinct edges.  To produce a maximal cell of the subspace $\Gamma^2_r$, we intersect $\Gamma^2_r$ with one of these cells of $\Gamma^2$.

For the first case, let $e$ be an edge with length $L_e$ and endpoints $a$ and $b$.  Then $\Gamma^2$ has two $2$--cells in which both $x$ and $y$ lie on $e$.  They can be identified with the simplices $\{(x,y) \in [0,L_e]^2 \, | \, x \leq y\}$ and $\{(x,y) \in [0,L_e]^2 \, | \, y \leq x\}$.  We will consider the intersection of the first of these (call it $c$) with $\Gamma^2_r$; the case of the second is symmetric.

Initially, the cell $c$ of $\Gamma^2$ is defined by the following inequalities in $\R^2$:
\begin{enumerate}
\item[$C^x_1$:] $-x \leq 0$,
\item[$C^x_2$:] $x \leq L_e$,
\item[$C^y_1$:] $-y \leq 0$,
\item[$C^y_2$:] $y \leq L_e$,
\item[$E^{xy}_1$:] $x - y \leq 0$.
\end{enumerate}
The inequalities labeled by $C^x_2$ and $C^y_1$ are redundant and can be ignored.

The cell $c_r$ of $\Gamma^2_r$ is the intersection of $c$ with $\Gamma^2_r$ and is given by the addition of the following two inequalities:
\begin{enumerate}
\item[$E^{xy}_2$:] $x - y \leq -r$,
\item[$E^{xy}_3$:] $-x + y \leq -r + \delta(a,b) + L_e$.
\end{enumerate}
The first of these ensures that the piece of $e$ between $x$ and $y$ has length at least $r$.  The second ensures that a shortest path from $x$ to $y$ passing through $a$ and $b$ has length at least $r$.   One of these two lengths is $\delta(x,y)$.

A candidate for a vertex is given by picking two of these inequalities whose left hand sides are linearly independent and changing them both to equations.  The solution to this pair of equations is a potential vertex.  To see if it is an actual vertex of $c_r$, we check to see if it satisfies all of the remaining inequalities.

\begin{figure}
\includegraphics[scale=.7]{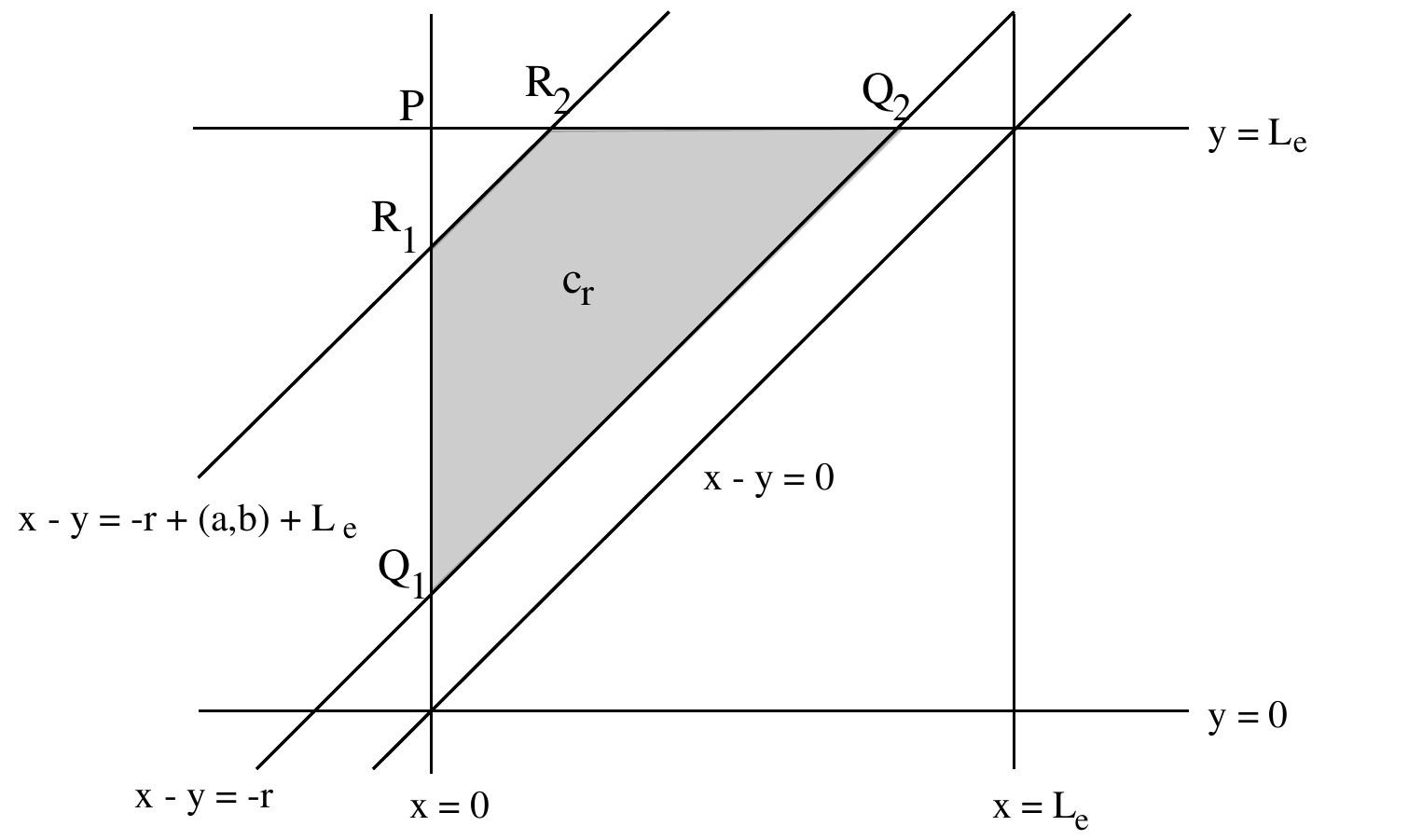}
\caption{$e \times e, x \leq y$ \label{fig1}}
\end{figure}

Choosing $C^x_1$ and $E^{xy}_1$ gives the candidate $(0,0)$, which satisfies the remaining inequalities exactly when $r = 0$.  Likewise, the candidate $(L_e,L_e)$ given by $E^{xy}_1$ and $C^y_2$ is an actual vertex only when $r = 0$.  

Choosing $C^x_1$ and $C^y_2$ gives $P = (0,L_e)$, which satisfies the others for all $r \leq \delta(a,b)$.  

Choosing $C^x_1$ and $E^{xy}_2$ gives $Q_1 = (0,r)$.  This point is an actual vertex for $0 \leq r \leq \frac{\delta(a,b) + L_e}{2}$.  Note that when $r = 0$, this vertex is the same as the vertex $(0,0)$ given by $C^x_1$ and $E^{xy}_1$.  Similarly, the point $Q_2 = (L_e - r, L_e)$ given by $C^y_2$ and $E^{xy}_2$ is a vertex on the same interval and is the same as $(L_e,L_e)$ when $r = 0$.  If $r = \frac{\delta(a,b) + L_e}{2}$ and it happens in $\Gamma$ that $\delta(a,b) = L_e$, then both $Q_1$ and $Q_2$ are the same vertex as $P$.

Finally, choosing $C^x_1$ and $E^{xy}_3$ gives the candidate $R_1 = (0, \delta(a,b) + L_e - r)$, which is a vertex when $\delta(a,b) \leq r \leq \frac{\delta(a,b) + L_e}{2}$.  We also have the vertex $R_2 = (r - \delta(a,b) ,L_e)$ on the same interval given by $C^y_2$ and $E^{xy}_3$.  When $r = \delta(a,b)$, these two vertices are the same as $P$, and when $r = \frac{\delta(a,b) + L_e}{2}$, they are identified with the respective vertices $Q_1$ and $Q_2$.

When $r > \frac{\delta(a,b) + L_e}{2}$, the intersection is empty, so there is no cell $c_r$ for $\Gamma^2_r$.  

Thus, for this type of cell, there are at most $2$ critical values of $r > 0$ that signify changes in the combinatorial type of $c_r$.  We remark that, for an edge $e$ having shortest length in $\Gamma$, the cell arising from $e \times e$ has only one critical value for $r$:  For such an edge, $\delta(a,b) = L_e$ and the potential vertices $R_1$ and $R_2$ in Figure \ref{fig1} will never be actual vertices separate from $Q_1$ and $Q_2$.

We now consider the other type of cell in $\Gamma^2$.  When the points $x$ and $y$ lie on distinct edges $e_1$ and $e_2$ respectively, $(x,y)$ is in the cell $c$ of $\Gamma^2$ defined by the following inequalities in $\R^2$:
\begin{enumerate}
\item[$C^x_1$:] $-x \leq 0$,
\item[$C^x_2$:] $x \leq L_{e_1}$,
\item[$C^y_1$:] $-y \leq 0$,
\item[$C^y_2$:] $y \leq L_{e_2}$.
\end{enumerate}
The cell $c_r = c \cap \Gamma^2_r$ is determined by those and the following (where $a_i$ and $b_i$ are the endpoints of $e_i$ corresponding to $0$ and $L_{e_i}$):
\begin{enumerate}
\item[$D^{xy}_1$:] $-x - y \leq -r + \delta(a_1,a_2)$,
\item[$D^{xy}_2$:] $-x + y \leq -r + \delta(a_1,b_2) + L_{e_2}$,
\item[$D^{xy}_3$:] $x - y \leq -r + \delta(b_1,a_2) + L_{e_1}$,
\item[$D^{xy}_4$:] $x + y \leq -r + \delta(b_1,b_2) + L_{e_1} + L_{e_2}$.
\end{enumerate}
The first of these ensures that any path from $x$ to $y$ passing through $a_1$ and $a_2$ has length longer than $r$.  The second ensures that one passing through $a_1$ and $b_2$ has sufficient length and so forth.  Together, the four inequalities guarantee that $\delta(x,y) \geq r$.

\begin{figure}
\includegraphics[scale=.7]{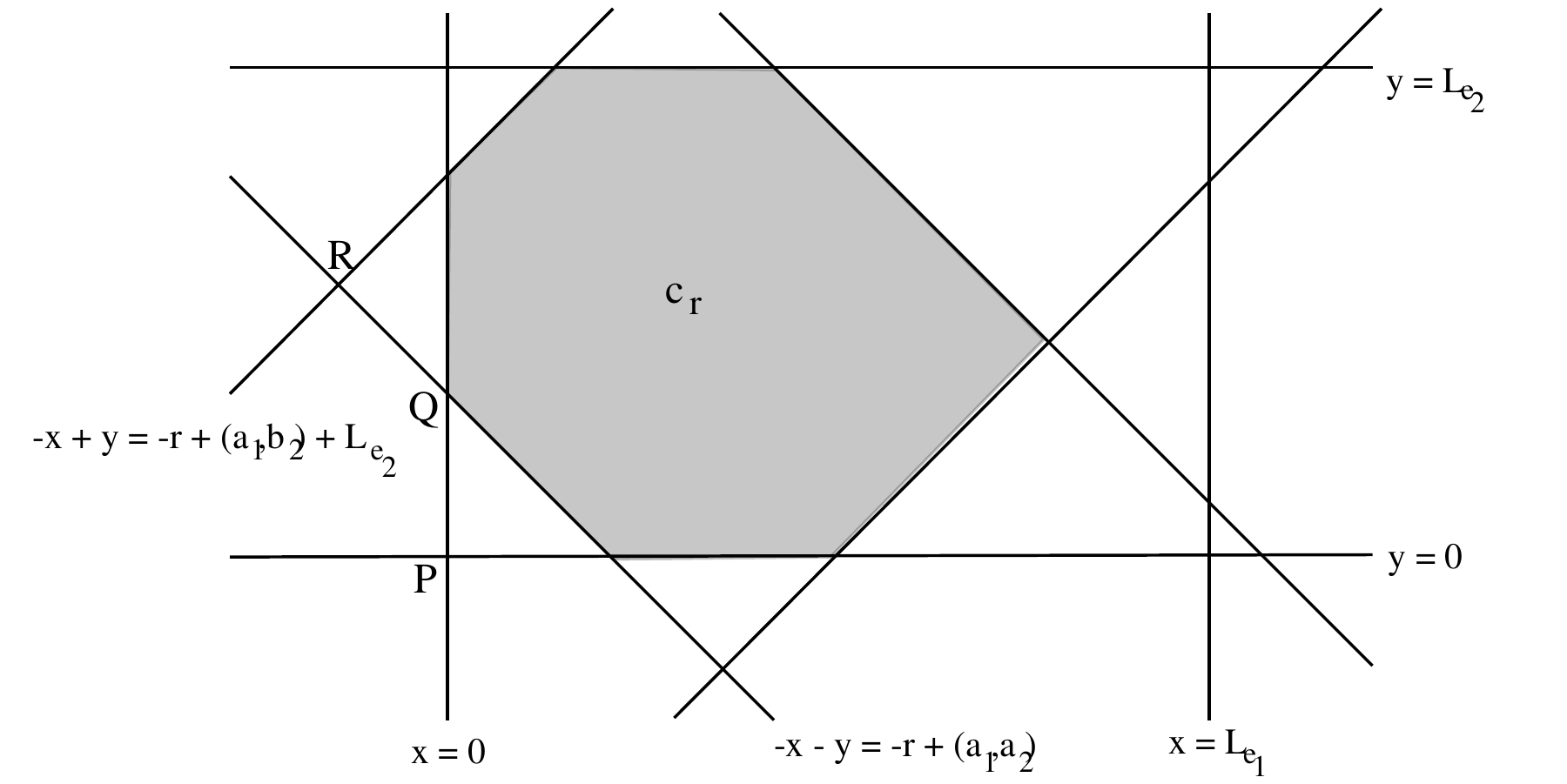}
\caption{$e_1 \times e_2$ \label{fig2}}
\end{figure}

There are three kinds of potential vertices of $c_r$, exemplified by $P$, $Q$, and $R$ in Figure \ref{fig2}.  The point $P = (0,0)$ is an actual vertex when $0 \leq r \leq \delta(a_1,a_2)$.  The point $Q = (0, r - \delta(a_1, a_2))$ is a vertex when $\delta(a_1,a_2) \leq r \leq \frac{\delta(a_1,a_2) + \delta(a_1,b_2) + L_{e_2}}{2}$, and $R = (r - \frac{\delta(a_1,a_2) + \delta(a_1,b_2) + L_{e_2}}{2},\frac{\delta(a_1,b_2) + L_{e_2} - \delta(a_1,a_2)}{2})$ is a vertex when $\frac{\delta(a_1,a_2) + \delta(a_1,b_2) + L_{e_2}}{2} \leq r \leq \min \{\frac{\delta(a_1,b_2) + \delta(b_2,a_1) + L_{e_1} + L_{e_2}}{2}, \frac{\delta(a_1,a_2) + \delta(b_1,b_2) + L_{e_1} + L_{e_2}}{2}\}$.  When $r = \delta(a_1,a_2)$, $P$ and $Q$ are identified.  When $r = \frac{\delta(a_1,a_2) + \delta(a_1,b_2) + L_{e_2}}{2}$, $Q$ and $R$ are identified.  We can analyze all of the other potential vertices by using the symmetries of interchanging $e_1$ and $e_2$ and/or reversing their orientations.  In total, there are at most $9$ critical values of $r$ which result in identifications between vertices and signify changes in the combinatorial type of $c_r$ as $r$ increases. The smallest of these will give the same type as $r = 0$.

We will now consider how the homeomorphism and homotopy types of $\Gamma^2_r$ change at a critical value of the restraint parameter $r$.  

\begin{figure}
\includegraphics[scale=.7]{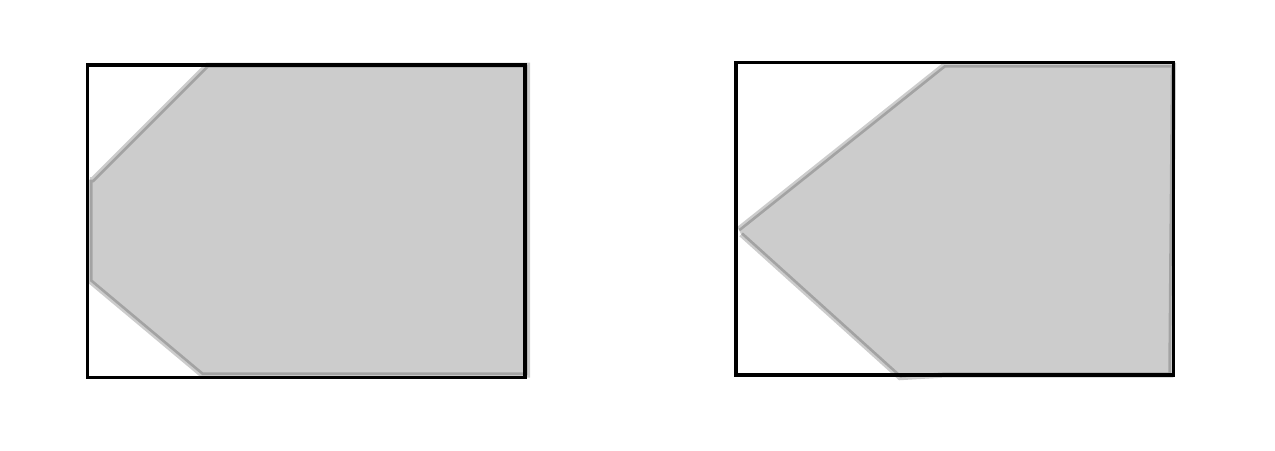}
\caption{A cell $c_r$  before and at a critical value \label{fig3}}
\end{figure}

The shaded regions in Figure \ref{fig3} represent the intersections of the rectangular cell $e_1 \times e_2$ of $\Gamma^2$ with two different restricted configuration spaces, the one on the right corresponding to a critical value for the restraint parameter and the one on the left corresponding to a slightly smaller value.  At this critical value, the part of the cell $c_r$ lying on $a_1 \times e_2$ (where $a_1$ is the initial node of the edge $e_1$) collapses to a single vertex.  The homeomorphism types of the two spaces are different, but the homotopy type remains the same.

\begin{figure}
\includegraphics[scale=.7]{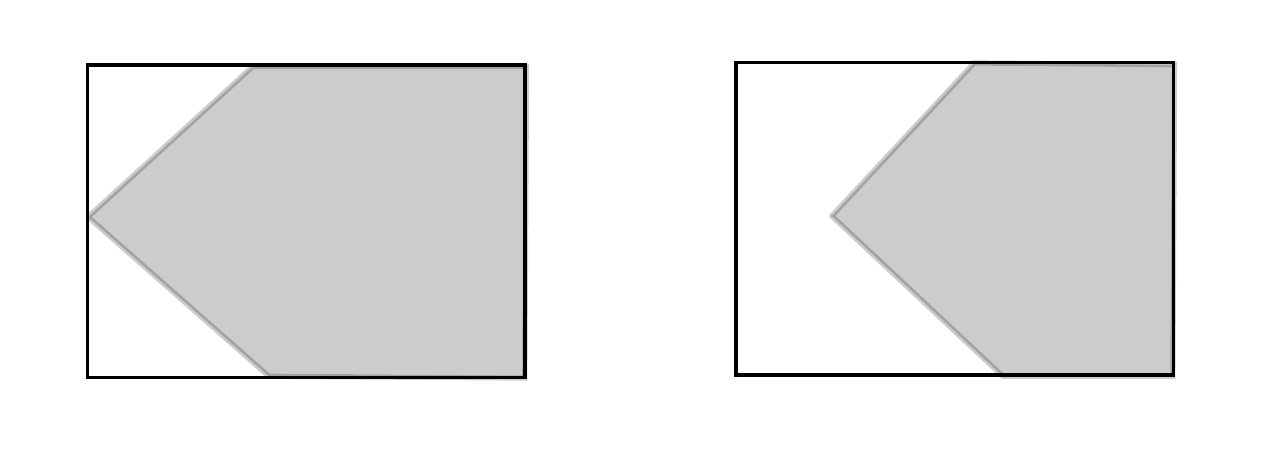}
\caption{A cell $c_r$  at and after a critical value \label{fig4}}
\end{figure}

In Figure \ref{fig4}, we see what happens to the cell $c_r$ immediately after the critical value.  The combinatorial type of this individual cell remains the same, but it no longer intersects $a_1 \times e_2$.  Both the homeomorphism and homotopy types (possibly) change.

\begin{figure}
\includegraphics[scale=.7]{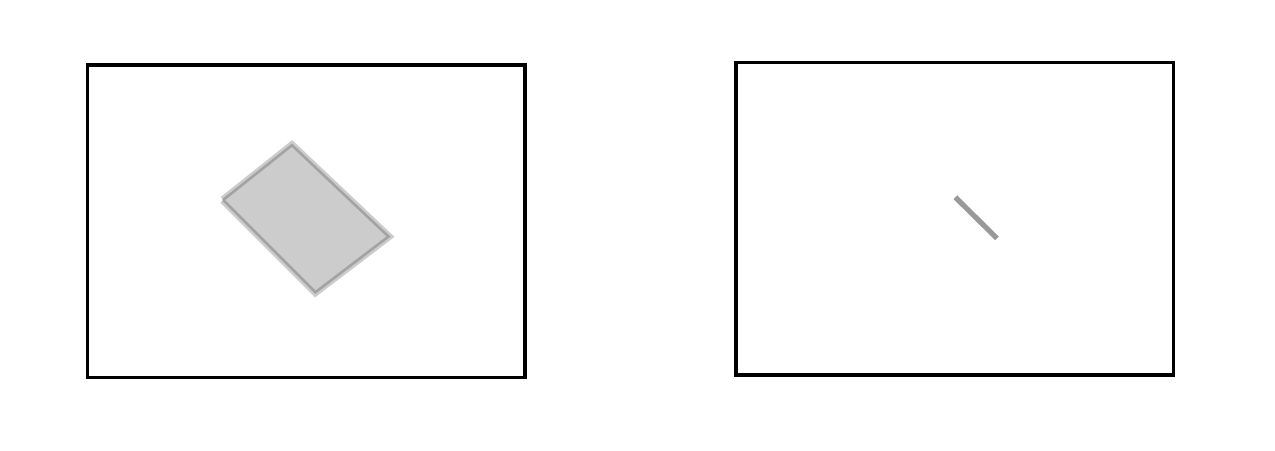}
\caption{Another type of critical value \label{fig5}}
\end{figure}

Figure \ref{fig5} shows the behavior of the cell $c_r$ at a larger critical value. Immediately before this critical value, the cell is already strictly contained by the interior of $e_1 \times e_2$, and it collapses to a lower dimension when $r$ reaches the critical value.  Again, the homeomorphism type changes, but the homotopy type does not.  Immediately after, the cell vanishes, changing the homotopy type.

As these examples show, the number of homotopy types for $\{ \Gamma^2_r\}_{r > 0}$ is bounded above by the number of critical values of $r$ (plus $1$ if we count the empty set as a homotopy type).  The number of homeomorphism (and isotopy) types is bounded above by twice the number of critical values (plus $1$ if we include the empty set).  Accordingly, we will now find a bound on the number of critical values of $r > 0$. 

For each edge $e$ of $\Gamma$, $e \times e$ is subdivided into two maximal cells, but their intersections with $\Gamma^2_r$ will have the same combinatorial type.  Thus, each edge of $\Gamma$ can contribute at most $2$ critical values as we saw in the first case of this section, and the shortest edge will only contribute $1$ critical value.  Likewise, the cells in $\Gamma^2_r$ coming from $e_1 \times e_2$ and $e_2 \times e_1$ will have the same type.  Thus, any choice of two distinct edges of $\Gamma$ can contribute at most $9$ critical values as we saw in the second case.  

If $E$ is the number of edges in $\Gamma$, the number of critical values of $r > 0$ is at most $9 \binom{E}{2} + 2E - 1 = \frac{9}{2}E^2 - \frac{5}{2}E - 1$.  Hence, including the empty set, the number of homotopy types for $\{\Gamma^2_r\}$ is bounded above by $\frac{9}{2}E^2 - \frac{5}{2}E$, and the number of isotopy types is bounded above by $9E^2 - 5E - 1$.

We also observe that the smallest nonzero critical value is the length $L_e$ of the shortest edge $e$ of $\Gamma$.  If $0 < r, s < L_e$, then $\Gamma^2_r$ is isotopic to $\Gamma^2_s$.  (Additionally, by our results in Section \ref{strict}, both are strong deformation retracts of the configuration space $\Gamma^{\underline{2}}$.)

\section{Polynomial Upper Bounds on the Isotopy Types of $\Gamma^n_{\vect{r}}$}\label{bound}

We can now establish the polynomial upper bound on the isotopy types of $\Gamma^n_{\vect{r}}$.  Let $\Omega$ be a $d$--dimensional affine subspace of the parameter space $\R^{\binom{n}{2}}$.  Then we have an affine map $\R^d \rightarrow \Omega \subseteq \R^{\binom{n}{2}}$.

In the full parameter space $\R^{\binom{n}{2}}$, different combinatorial types of a cell $c_{\vect{r}}$ are separated by hyperplanes of the form: \begin{equation*}A_{\lambda}\vect{A}_{\beta}^{-1}\vect{b}_{\beta}(\vect{r}) = b_{\lambda}(\vect{r})\end{equation*} where $\beta \subseteq \Lambda^c$ corresponds to some basis of $\R^n$ and $\lambda \notin \beta$.  If we restrict the parameters to $\Omega$, different combinatorial types are separated by preimages of these hyperplanes in $\R^d$.  Thus, we can find an upper bound to the number of types for $c_{\vect{r}}$ by counting the number of faces in the arrangement of all such hyperplanes in the parameter space $\R^d$.

If $\Gamma$ has $E$ edges, then $\Gamma^n$ has $E^{\bar{n}} = E(E+1)\cdots(E+n-1)$ $n$--cells.  If $c$ is an $n$--cell of $\Gamma^n$, recall that  $\{A_{\lambda}x\}_{\lambda \in \Lambda^c} \subseteq \{\pm x_i\}^n_{i=1} \cup \{\pm x_i \pm x_j\}_{1 \leq i < j \leq n}$.  Each $n$--cell $c$ contributes hyperplanes indexed by pairs $(\beta, \lambda)$, where $\beta \subset \Lambda^c$ corresponds to a basis for $\R^n$ and $\lambda \in \Lambda^c \setminus \beta$.  Therefore, an $n$-cell $c$ of $\Gamma^n$ will contribute at most $h_n := 2^n \binom{n^2}{n}(2n^2 - n)$ hyperplanes, so in all, there are $h_n E^{\bar{n}}$ hyperplanes in the parameter space $\R^d$.

\begin{lemma}
Let $f_d(k)$ be the number of faces of an arrangement of $k$ hyperplanes in $\R^d$.  Then $f_d(k) \leq 2^d \binom{k}{d} + 2^{d-1}\binom{k}{d-1} + \cdots + 2^0 \binom{k}{0}$.
\end{lemma} 
\begin{proof}
We use induction on $k + d$, $d > 0, k \geq 0$.  For $d > 0$, we have $f_d(0) = 1 = 2^d\binom{0}{d} + \cdots + 2^1\binom{0}{1} + 2^0\binom{0}{0}$.

Suppose $\mathcal{H}_{d}(k+1)$ is an arrangement of $k + 1$ hyperplanes in $\R^d$ with $f_d(k+1)$ faces.  Choose one hyperplane $H$.  Removing $H$ produces an arrangement $\mathcal{H}_{d}(k)$ of $k$ hyperplanes with $f_d(k)$ faces.  Also, intersecting each hyperplane of $\mathcal{H}_d(k)$ with $H$ induces an arrangement $\mathcal{H}_{d-1}(k)$ of (at most) $k$ hyperplanes in $\R^{d-1}$ with $f_{d-1}(k)$ faces.  The hyperplane $H$ will separate some faces of $\mathcal{H}_d(k)$ into three faces of $\mathcal{H}_d(k+1)$.   Therefore, $f_d(k+1) \leq f_d(k) + 2f_{d-1}(k)$.  By the induction hypothesis, we have that $f_d(k+1) \leq (2^d\binom{k}{d} + \cdots + 1) + 2(2^{d-1}\binom{k}{d-1} + \cdots + 1) = 2^d \binom{k+1}{d} + \cdots + 2\binom{k+1}{1} + 1$.  
\end{proof}

Setting $k = h_n E^{\bar{n}}$, we have that $f_d(h_n E^{\bar{n}})$ is a polynomial in $E$ of degree $n d$.  This yields the following:

\begin{theorem}\label{maintheorem}
Let $\Omega$ be a $d$--dimensional affine subspace of $\R^{\binom{n}{2}}$.  The number of isotopy types of $\{\Gamma^n_{\vect{r}}\}_{\vect{r} \in \Omega}$ is bounded above by a polynomial of degree $n d$ in the number of edges of $\Gamma$.
\end{theorem}

In particular, the number of isotopy types in $\{\Gamma^n_r\}_{r\geq 0}$ is bounded above by a  polynomial of degree $n$.  When $n = 2$, the following example gives an infinite family of metric graphs with a quadratic lower bound, showing that we cannot do better.

\begin{figure}
\includegraphics[scale=.7]{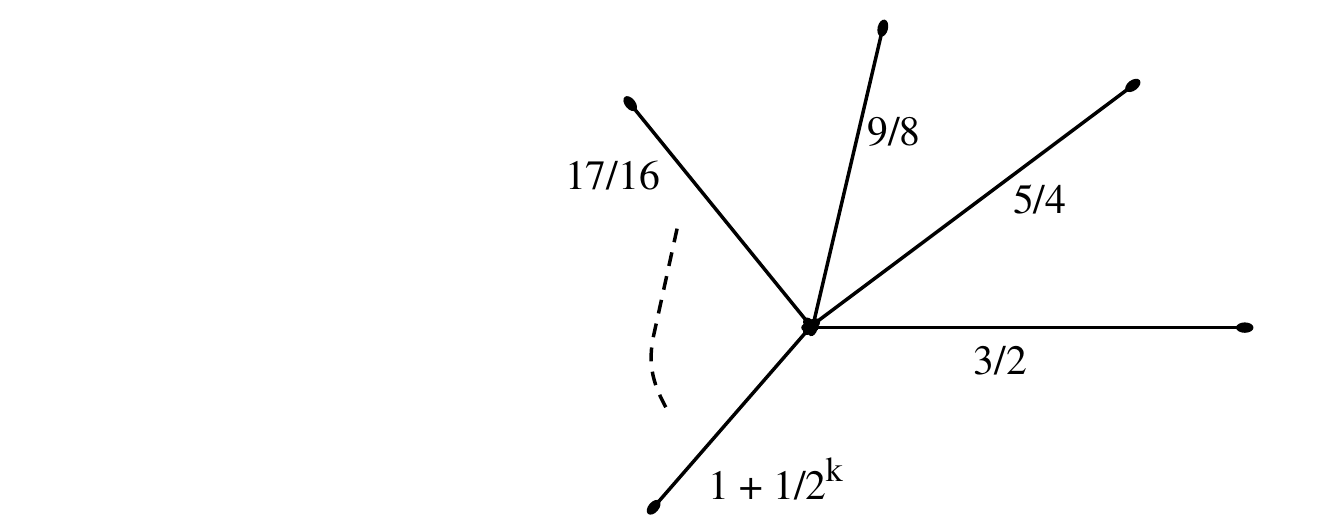}
\caption{$\Gamma$, the $k$--corolla \label{fig6}}
\end{figure}

\begin{exa}\label{corolla}
Let $\Gamma$ be the tree with $k$ edges, $e_1, \ldots, e_k$, emanating from a single node with lengths given by $L_{e_i} = 1 + \frac{1}{2^i}$.  Then $\{\Gamma^2_r\}_{r\geq0}$ has at least $\binom{k}{2}$ homotopy types (and hence isotopy types).
\end{exa}
\begin{proof}
We restrict ourselves to $r > L_{e_1} = \frac{3}{2}$.  A configuration $(x_1,x_2) \in \Gamma^2_r$ must have the two particles $x_1$ and $x_2$ positioned on distinct edges.  Neither particle can be moved through the common node while the configuration remains in $\Gamma^2_r$.  Thus, any nonempty maximal cell of $\Gamma^2_r$ is a path component.  

We now determine how many such path components there are.  A cell $c_r = (e_i \times e_j) \cap \Gamma^2_r$ is nonempty if and only if $r \leq L_{e_i} + L_{e_j}$.  Since, by construction, each sum $L_{e_i} + L_{e_j}$ is distinct and greater than $L_{e_1}$, different choices of $\{i,j\}$ yield distinct numbers of path components (and therefore distinct homotopy types) for $\Gamma^2_{L_{e_i} + L_{e_j}}$.  
\end{proof}

\section{Strict Proximity Restraints}\label{strict}

Let $I$ be a fixed collection of pairs $\{i, j\}$ from $\{1, \ldots, n\}$.  We now define the space:

\begin{equation*}
{}_I \Gamma^n_{\vect{r}} := \left\{(x_1, \ldots, x_n) \in \Gamma^n \, | \begin{array}{lr} \delta(x_i,x_j) \geq r_{ij}, & \{i,j\} \notin I\\
\delta(x_i,x_j) > r_{ij}, & \{i,j\} \in I 
\end{array}\right \}.
\end{equation*}

Define $\Omega_I := \left\{\vect{u} \in \R^{\binom{n}{2}} \, | \begin{array}{lr} u_{ij} = 0, & \{i,j\} \notin I\\
u_{ij} > 0, & \{i,j\} \in I 
\end{array}\right \}.$

For given $\vect{r} \in \Omega$ and $\vect{u} \in \Omega_I$, define a map $d = d_{\vect{r},\vect{u}} \colon {}_I\Gamma^n_{\vect{r}} \rightarrow (0,\infty)$ by $d(x) := \min_{\{i,j\}\in I} \frac{\delta(x_i,x_j) - r_{ij}}{u_{ij}}$.

\begin{theorem}\label{DefRetract}
Let $\vect{u} \in \Omega_I$.  Then, for sufficiently small $\epsilon > 0$, $\Gamma^n_{\vect{r} + \epsilon \vect{u}}$ is a strong deformation retract of ${}_I \Gamma^n_{\vect{r}}$.
\end{theorem}

\begin{proof}
Consider the ray $\vect{r} + \kappa \vect{u}$ in $\Omega$.  If ${}_I\Gamma^n_{\vect{r}} \neq \emptyset$, then there exists $\kappa^{\prime} > 0$ such that $\Gamma^n_{\vect{r} + \kappa \vect{u}} \neq \emptyset$ for $0 < \kappa \leq \kappa^{\prime}$.  Furthermore, since there are only finitely many equivalence classes on $\Omega$ and the intersection of the ray with any equivalence class is convex (since the components of $\vect{u}$ are nonnegative), we may also find $\epsilon < \kappa^{\prime}$ such that for all $0 < \kappa \leq \epsilon$, $\vect{r} + \kappa \vect{u} \sim \vect{r} + \epsilon \vect{u}$.  Fix this $\epsilon$.

We now define the deformation retraction $\Psi \colon {}_I\Gamma^n_{\vect{r}} \times [0,1] \rightarrow {}_I\Gamma^n_{\vect{r}}$ by:

\begin{equation*}
\Psi(x,t) := \left\{\begin{array}{lr} (1-t)x + t\phi^{\vect{r} + d(x) \vect{u}}_{\vect{r}+ \epsilon \vect{u}}(x), & d(x) \leq \epsilon \\
x, & d(x) \geq \epsilon 
\end{array}\right.
\end{equation*}

Note that $d^{-1}[\epsilon,\infty) = \Gamma^n_{\vect{r} + \epsilon \vect{u}}$ is fixed for all $t$.  We will now prove that $x \mapsto \phi^{\vect{r} + d(x) \vect{u}}_{\vect{r}+ \epsilon \vect{u}}(x)$ is continuous, and the continuity of $\Psi$ will be clear.

Let $f \colon d^{-1}(0,\epsilon] \rightarrow d^{-1}(\epsilon)$ be given by $f(x) = \phi^{\vect{r} + d(x) \vect{u}}_{\vect{r}+ \epsilon \vect{u}}(x)$.  This is well defined:  By the definition of $d(x)$, there exists some $\{i, j\} \in I$ such that $\delta(x_i,x_j) = r_{ij} + d(x)u_{ij}$ and for any $\{k,l\} \in I$, $\delta(x_k,x_l) \geq r_{kl} + d(x) u_{kl}$.  Applying $\phi^{\vect{r} + d(x) \vect{u}}_{\vect{r}+ \epsilon \vect{u}}$ to $x$, we have the same equalities/inequalities but with $\epsilon$ taking the place of $d(x)$.  Hence, $d(f(x)) = \epsilon$.

Now let $U$ be an open set in $d^{-1}(\epsilon)$.  Let $x_0 \in f^{-1}(U)$.  Then $f(x_0) = \phi^{\vect{r} + d(x_0) \vect{u}}_{\vect{r}+ \epsilon \vect{u}}(x_0) = y_0 \in U$ and $d(x_0) \leq \epsilon$.  Choose $\kappa \in (0,d(x_0))$.

Define $\Phi \colon d^{-1}(\epsilon) \times [0, 1] \rightarrow d^{-1}[\kappa,\epsilon]$ by $\Phi(y,t) = (1-t)y + t \phi^{\vect{r} + \epsilon \vect{u}}_{\vect{r} + \kappa \vect{u}}(y)$.  Then, as in an earlier theorem, $\Phi(-,t)$ is the homeomorphism $\phi^{\vect{r} + \epsilon \vect{u}}_{\vect{t}}$ where $\vect{t} = \vect{r} + ((1-t)\epsilon + t\kappa)\vect{u}$.  Also, $d(\Phi(y,t)) = (1-t)\epsilon + t\kappa$.  It easily follows that $\Phi$ is a homeomorphism.

We now claim that $x_0 \in \Phi(U \times [0,1)) \subseteq f^{-1}(U)$ which will show that $f^{-1}(U)$ is open and $f$ is continuous.
We have that $d(x_0) = (1-t_0)\epsilon + t_0 \kappa$ for $t_0 = \frac{\epsilon - d(x_0)}{\epsilon - \kappa}$.  Then $\Phi(y_0,t_0) = \phi^{\vect{r} + \epsilon \vect{u}}_{\vect{r} + d(x_0) \vect{u}}(y_0) =  (\phi^{\vect{r} + d(x_0) \vect{u}}_{\vect{r}+ \epsilon \vect{u}})^{-1}(y_0) = x_0$.  Thus, $x_0 \in \Phi(U \times [0,1))$.

Now let $(y,t) \in U \times [0,1)$.  Then $\Phi(y,t) = \phi^{\vect{r} + \epsilon \vect{u}}_{\vect{t}}(y)$ so that $f(\Phi(y,t)) = \phi^{\vect{t}}_{\vect{r} + \epsilon \vect{u}}(\phi^{\vect{r} + \epsilon \vect{u}}_{\vect{t}}(y)) = y \in U$.  Therefore, $\Phi(U \times [0,1) \subseteq f^{-1}(U)$.
\end{proof}

Since the configuration space $\Gamma^{\underline{n}}$ is the same as the space ${}_I\Gamma^n_0$, where the set $I$ consists of all pairs $\{i,j\}$, we have the following immediate corollary:

\begin{cor}\label{ConfigSpaceDefRetracts}
There exists $s > 0$ such that, for all $0 < r < s$, $\Gamma^n_r$ is a strong deformation retract of $\Gamma^{\underline{n}}$.
\end{cor}

Finally, we give the following result regarding the homeomorphism types of the restricted configuration spaces with strict proximity restraints:

\begin{theorem}\label{Open-Homeo}
Let $\vect{r}, \vect{s} \in \Omega$.  Suppose there exist $\vect{u}$ and $\vect{v}$ in $\Omega_I$ and $\epsilon > 0$ such that, for $0 < \kappa < \epsilon$,  $\vect{r} + \kappa \vect{u} \sim \vect{r} + \epsilon \vect{u} \sim \vect{s} + \epsilon \vect{v} \sim \vect{s} + \kappa \vect{v}$.  Then ${}_I \Gamma^n_{\vect{r}}$ is homeomorphic to ${}_I \Gamma^n_{\vect{s}}$.  In fact, they are isotopic in $\Gamma^n$.
\end{theorem}

We remark that the hypothesis regarding $\epsilon$ simply means that, using the construction in the proof of \ref{DefRetract}, ${}_I \Gamma^n_{\vect{r}}$ and ${}_I \Gamma^n_{\vect{s}}$ deformation retract to closed spaces of the same isotopy type.

\begin{proof}
Pick $\vect{u}, \vect{v}$, and $\epsilon$ as in the statement.  Let $d = d_{\vect{r},\vect{u}}$ and $d^{\prime} = d_{\vect{s},\vect{v}}$.  
We define the homeomorphism $\theta = \theta^{\vect{r},\vect{u}}_{\vect{s},\vect{v}} = \theta \colon {}_I\Gamma^n_{\vect{r}} \rightarrow {}_I\Gamma^n_{\vect{s}}$ by:

\begin{equation*}
\theta(x) := \left\{\begin{array}{lr} \phi^{\vect{r} + d(x) \vect{u}}_{\vect{s}+ d(x) \vect{v}}(x), & d(x) \leq \epsilon \\
\phi^{\vect{r} + \epsilon \vect{u}}_{\vect{s} + \epsilon \vect{v}}(x), & d(x) \geq \epsilon 
\end{array}\right.
\end{equation*}

We remark that, if $d(x) \leq \epsilon$, we have that $d^{\prime}(\phi^{\vect{r}+d(x)\vect{u}}_{\vect{s} + d(x)\vect{v}}(x)) = d(x)$.  Therefore, the inverse map is given by:

\begin{equation*}
\theta^{-1}(x) := \left\{\begin{array}{lr} \phi^{\vect{s} + d^{\prime}(x) \vect{v}}_{\vect{r}+ d^{\prime}(x) \vect{u}}(x), & d^{\prime}(x) \leq \epsilon \\
\phi^{\vect{s} + \epsilon \vect{v}}_{\vect{r} + \epsilon \vect{u}}(x), & d^{\prime}(x) \geq \epsilon 
\end{array}\right.
\end{equation*}

Therefore, we have that $\theta$ is one-to-one and onto.  All that remains is the continuity of these maps which will be proven if we show that the map $x \mapsto \phi^{\vect{r} + d(x) \vect{u}}_{\vect{s}+ d(x) \vect{v}}(x)$ is continuous.

First, consider the map $g = g_{\vect{r},\vect{u}} \colon d^{-1}(\epsilon) \times (0,\epsilon] \rightarrow d^{-1}(0,\epsilon]$ given by $g(x,t) = \phi^{\vect{r} + \epsilon \vect{u}}_{\vect{r} + t \vect{u}}(x)$.  This is one-to-one:  $d(g(x,t)) = t$ and $g(-,t) =\phi^{\vect{r} + \epsilon \vect{u}}_{\vect{r} + t \vect{u}}$ is a homeomorphism.  Also, $g$ is onto:  For $y \in d^{-1}(0,\epsilon]$, $g^{-1}(y) = (\phi^{\vect{r} + d(y)\vect{u}}_{\vect{r}+\epsilon \vect{u}}(y), d(y)) = (f(y),d(y))$.  Notice here that $g^{-1}$ is continuous because both component functions are continuous.

For $\kappa \in (0,\epsilon]$, let $g_{\kappa}$ be the restriction of $g$ to $d^{-1}(\epsilon) \times [\kappa,\epsilon]$.  Then $g_{\kappa}^{-1}$ is a continuous bijection from the compact $d^{-1}[\kappa,\epsilon]$ to the Hausdorff $d^{-1}(\epsilon) \times [\kappa,\epsilon]$, so $g_{\kappa}$ is a homeomorphism.  

Now, let $U \subseteq d^{-1}(0,\epsilon]$ be open and $(x,t) \in g^{-1}(U)$.  Choose $\kappa < t$.  Then $V = U \cap d^{-1}(\kappa,\epsilon]$ is open in $d^{-1}(0,\epsilon]$, and $g^{-1}(V) = g^{-1}_{\kappa}(V)$ is open in $d^{-1}(\epsilon) \times (\kappa, \epsilon]$, and so too in  $d^{-1}(\epsilon) \times (0, \epsilon]$.  Then $(x,t) \in g^{-1}(V) \subset g^{-1}(U)$, and we have that $g$ is continuous.  

We now consider the map $x \mapsto \phi^{\vect{r} + d(x) \vect{u}}_{\vect{s}+ d(x) \vect{v}}(x)$.  This map can be written as a composition of continuous functions: $x \mapsto g_{\vect{s},\vect{v}}(\phi^{\vect{r} + \epsilon \vect{u}}_{\vect{s} + \epsilon \vect{v}}(\phi^{\vect{r} + d(x) \vect{u}}_{\vect{r} + \epsilon \vect{u}}(x)), d(x))$.  Hence, $\theta$ and $\theta^{-1}$ are continuous.

Next, we define the isotopy.  
Define $\Theta \colon {}_I\Gamma^n_{\vect{r}} \times [0,1] \rightarrow \Gamma^n$ by $\Theta(x,t) = (1-t)x + t \theta^{\vect{r},\vect{u}}_{\vect{s},\vect{v}} (x)$.  It can be shown that, for $d(x) \leq \epsilon$, $\Theta(x,t) = \phi^{\vect{r} + d(x)\vect{u}}_{((1-t)\vect{r} + t\vect{s}) + d(x)((1-t)\vect{u} + t\vect{v})}(x)$.  For $d(x) \geq \epsilon$, $\Theta(x,t) = \phi^{\vect{r} + \epsilon \vect{u}}_{((1-t)\vect{r} + t\vect{s}) + \epsilon((1-t)\vect{u} + t \vect{v})}$.  Together, that means that $\Theta(-,t) = \theta^{\vect{r},\vect{u}}_{(1-t)\vect{r} + t\vect{s},(1-t)\vect{u} + t\vect{v}}$.  

As before, we note that, for given $t \in (0,1)$ and $\kappa \in (0,\epsilon]$, $(1-t)\vect{r} + t\vect{s} + \kappa((1-t)\vect{u} + t\vect{v})$ may not be in the equivalence class of $\vect{r} + \kappa \vect{u}$.  However, we can define the homeomorphisms to construct $\theta^{\vect{r},\vect{u}}_{(1-t)\vect{r} + t\vect{s},(1-t)\vect{u} + t\vect{v}}$ on just the cells $c$ of $\Gamma^n$ for which $c_{\vect{r}}$ is nonempty.

Thus, we see that $\Theta(-,t)$ is a homeomorphism, so $\Theta$ is an isotopy.
\end{proof}





\bibliographystyle{model1b-num-names}







\end{document}